\documentclass[reqno,11pt]{amsart}

\usepackage[T1]{fontenc} 
\usepackage[utf8]{inputenc} 
\usepackage{mathtools}
\usepackage{amssymb}
\usepackage{amsthm}
\usepackage{bbm}
\usepackage{mathrsfs}
\usepackage[english]{babel} 
\usepackage{enumitem}
\usepackage{dsfont}
\usepackage{url}

\usepackage{hyperref}
\hypersetup{linktocpage=true, colorlinks=true, breaklinks=true, linkcolor=blue, menucolor=black, urlcolor=black,citecolor=blue,filecolor=green}

\newcommand{\N}{\mathbb{N}} % natuerliche Zahlen
\newcommand{\R}{\mathbb{R}} % reelle Zahlen

\newcommand{\1}{\mathds{1}}

\theoremstyle{plain}
\newtheorem{thrm}{Theorem}[section]
\newtheorem{prop}[thrm]{Proposition}
\newtheorem{cor}[thrm]{Corollary}
\newtheorem{lemma}[thrm]{Lemma}

\theoremstyle{definition}

\newtheorem{remark}[thrm]{Remark}

\def\N{\bbN}
\numberwithin{equation}{section}
\DeclareMathSymbol{\leqslant}{\mathalpha}{AMSa}{"36}
\DeclareMathSymbol{\geqslant}{\mathalpha}{AMSa}{"3E}
\DeclareMathSymbol{\doteqdot}{\mathalpha}{AMSa}{"2B}
\DeclareMathSymbol{\circlearrowright}{\mathalpha}{AMSa}{"08}
\DeclareMathSymbol{\subsetneq}{\mathalpha}{AMSb}{"28}
\DeclareMathSymbol{\supsetneq}{\mathalpha}{AMSb}{"29}
\renewcommand{\leq}{\;\leqslant\;}
\renewcommand{\geq}{\;\geqslant\;}

\newcommand{\upchi}{\raise 2pt \hbox{$\chi$}}
\def\writefig#1 #2 #3 {\rlap{\kern #1 truecm \raise #2 truecm
		\hbox{#3}}}

\newcommand{\bbN}{{\mathbb N}}

\begin{document}
	%{\hfill\small \version} \vspace{2mm}
	
	\title{Effective mass of the Polaron: a lower bound}
	
	\author{Volker Betz}
	\address{Volker Betz \hfill\newline
		\indent FB Mathematik, TU Darmstadt}
	\email{betz@mathematik.tu-darmstadt.de}
	
	\author{Steffen Polzer}
	\address{Steffen Polzer \hfill\newline
		\indent Section de mathématiques, Université de Genève
	}
	\email{steffen.polzer@unige.ch}
	
	\maketitle
	\begin{quote}
	{\small
		{\bf Abstract.}
		We show that the effective mass of the Fröhlich Polaron is bounded below by $c \alpha^{2/5}$ for some 
		constant $c>0$ and  for all coupling constants $\alpha$.
		The proof uses the point process representation of the path measure of the Fröhlich Polaron.
	} 	
\end{quote}
\section{Introduction and results}
The Polaron models the slow movement of an electron in a polar crystal. The electron drags a cloud of polarization along and thus appears heavier, it has an ``effective mass'' larger then its mass without the interaction.
In the Fröhlich model of the Polaron, the Hamiltonian describing the interaction of the electron with the lattice vibrations is the operator given by
\begin{equation*}
	H_\alpha = \frac12 p^2 \otimes \1 + \1\otimes \mathbf N + \frac{\sqrt{\alpha}}{\sqrt{2} \pi} \int_{\mathbb R^3} 
	\frac{1}{|k|}\big(e^{i k \cdot x} a_k + e^{-i k \cdot x} a^\ast_k \big) \, 
	\mathrm d k
\end{equation*}
that acts on $L^2(\R^3) \otimes \mathcal F$ (where $\mathcal F$ is the bosonic Fock space over $L^2(\R^3)$).
Here $x$ and $p$ are the position and momentum operators of the electron, $\mathbf N \equiv  \int_{\mathbb R^3} a^\ast_k a_k\, \mathrm d k$ is the number operator, the creation and annihilation operators $a_k^\ast$ and $a_k$ satisfy the canonical commutation relations $[a^\ast_k, a_{k'}] = \delta(k-k')$ and $\alpha>0$ is the coupling constant.
The Hamiltonian commutes with the total momentum operator $p\otimes\1 + \1 \otimes P_f$, where $P_f \equiv \int_{\R^3} k a_k^* a_k \mathrm dk$ is the momentum operator of the field, and has a fiber decomposition in terms of the Hamiltonians
\begin{equation*}
	H_\alpha(P) = \frac12 (P - P_f)^2 + \mathbf N + \frac{\sqrt{\alpha}}{\sqrt{2} \pi} \int_{\mathbb R^3} 
	\frac{1}{|k|}(a_k + a^\ast_k) \, \mathrm dk
\end{equation*}
at fixed total momentum $P$ (which act on $\mathcal F$). Of particular interest is the energy-momentum relation $E_\alpha(P) \coloneqq \inf \operatorname{spec}(H_\alpha(P))$. It is known that $E_\alpha$ is rotationally symmetric and has a global minimum at $P=0$, which is believed to be unique \cite{DySp20}. The effective mass $m_{\text{eff}}(\alpha)$ is defined as the inverse curvature at $P=0$ i.e.
 \begin{equation*}
 	E_\alpha(P) - E_\alpha(0) = \frac{1}{2m_{\text{eff}}(\alpha)} |P|^2 + o(|P|^2).
 \end{equation*}
While the asympotics of the ground state energy $E_\alpha(0)$ in the strong coupling limit $\alpha \to \infty$ are known, the asymptotics of the effective mass $m_\text{eff}(\alpha)$ still remain an important open problem. For the ground state energy, Donsker and Varadhan \cite{DoVa83} gave a probabilistic proof (using the path integral formulation) of Pekars conjecture \cite{Pe49} that states
\begin{equation}
	\label{Equation: Pekhar variational formula}
	\lim_{\alpha \to \infty} \frac{E_\alpha(0)}{\alpha^2} = \min_{\psi \in H^1, \|\psi\|_{L^2} = 1} \Big[ \| \nabla \psi\|^2_{L_2} - \sqrt{2}\iint_{\R^3 \times \R^3} \mathrm dx \mathrm dy \, \frac{|\psi(x) \psi(y)|^2}{|x-y|} \Big].
\end{equation}
There also exists a functional analytic proof by Lieb and Thomas \cite{LiTh97}, that additionally contains explicit error bounds.
While it has been conjetured by Landau and Pekar \cite{LP48} that
\begin{equation}
	\label{Equation: Asymptotics effective mass}
	\lim_{\alpha \to \infty} \frac{m_{\text{eff}}(\alpha)}{\alpha^4} = \frac{16 \sqrt 2\pi}{3} \int_{\R^3} \mathrm dx \, |\psi(x)|^4, \quad  \psi \text{ minimizer of \eqref{Equation: Pekhar variational formula}},
\end{equation}
so far the only progress on the asympotics of the effective mass has been achieved by Lieb and Seiringer \cite{LiSe20}, who showed that
$
	\lim_{\alpha \to \infty} m_\text{eff}(\alpha) = \infty	
$
by estimating
\begin{equation*}
	\frac{1}{2m_\text{eff}(\alpha)} \leq  \lim_{P\to 0} \frac{1}{|P|^2} \big( \langle \phi_{\alpha, P}, H_\alpha(P) \phi_{\alpha, P} \rangle - E_\alpha(0)\big)
\end{equation*}
 with suitably chosen trial states $\phi_{\alpha, P}$.
For models with stronger regularity assumptions (excluding the Fröhlich Polaron), quantitative estimates on the effective mass were recently obtained in \cite{MS21}. Our goal is to show that there exists some constant $c>0$ such that for all $\alpha>0$
\begin{equation}
\label{Equation: Our growth estimate}
	m_\text{eff}(\alpha) \geq c a^{2/5}
\end{equation}
and thereby giving a first quantitative growth bound for the effective mass of the Fröhlich Polaron.\\
\\
We give a short introduction into the probabilistic representation of the effective mass. We refer the reader to \cite{DySp20} for details in this representation, and to 
\cite{Moe06} for a review covering functional analytic properties of the Polaron.
An application of the Feynman-Kac formula to the semigroup $(e^{-TH_\alpha})_{T \geq 0}$ leads to the path measure 
\begin{equation}
\label{Equation: Definition of our Gibb measure}
\mathbb P_{\alpha, T}(\mathrm d\mathbf x) = \frac{1}{Z_{\alpha, T}}\exp\bigg( \frac{\alpha}{2} \int_{0}^T\int_{0}^T \frac{e^{-|t-s|}}{|\mathbf x_{t}- \mathbf x_{s}|} \, \mathrm ds \mathrm dt \bigg) \mathcal W(\mathrm d \mathbf x)
\end{equation}
in finite volume $T>0$, where $\mathcal W$ is the distribution of three dimensional Brownian motion and the partition function $Z_{\alpha, T}$ is a normalization constant.
In particular, the partition function can be expressed as the matrix element $Z_{\alpha, T} = \langle \Omega, e^{-T H_\alpha(0)} \Omega \rangle$, where $\Omega$ is the Fock vacuum. This leads to $E_\alpha(0) = - \lim_{T\to \infty} \log(Z_{\alpha, T})/T$, which was applied (up to boundary conditions) in the aforementioned proof of Pekars conjecture by Donsker and Varadhan.
Spohn conjectured \cite{Sp87} that the path measure (in infinite volume) converges under diffusive rescaling to Brownian motion with some diffusion constant $\sigma^2(\alpha)$, and showed that the effective mass then has the representation
\begin{equation*}
	m_{\text{eff}}(\alpha) = (\sigma^2(\alpha))^{-1}.
\end{equation*} The existence of the infinite volume measure and the validity of the central limit theorem were shown by Mukherjee and Varadhan \cite{MV19} for a restricted range of coupling parameters. The proof relies on a representation of the measures \eqref{Equation: Definition of our Gibb measure} as a mixture of Gaussian measures, the mixing measure being the distribution of a point process on $\{(s, t)\in \R^2: s<t\} \times (0, \infty)$ which has a renewal structure.
This approach was extended in \cite{BP21} to a broader class of path measures and a functional central limit theorem, and, for the Fröhlich Polaron, to all coupling parameters (relying on known spectral properties of $H_\alpha(0)$). We also refer the reader to \cite{MV21}, \cite{BMPV22}, where proofs are given that do not need the spectral properties of $H_\alpha(0)$. We use this point process representation and derive a ``variational like'' formula for the effective mass (see Proposition \ref{Proposition: Point process representation of diffusion constant}). We obtain the estimate \eqref{Equation: Our growth estimate} by minimizing over sub-processes of the full point process.\\
\\
Without additional effort, we can allow for a bit more generality and look at the probability measures $\mathbb P_{\alpha, T}$ on $C([0, \infty), \R^d)$ defined by
\begin{equation}
\label{Equation: Definition of path measures}
\mathbb P_{\alpha, T}(\mathrm d\mathbf x) = \frac{1}{Z_{\alpha, T}}\exp\bigg( \frac{\alpha}{2} \int_{0}^T\int_{0}^T g(t-s) v(\mathbf x_{s, t}) \, \mathrm ds \mathrm dt \bigg) \mathcal W(\mathrm d \mathbf x)
\end{equation}
where $d\geq 2$, $v(x) = 1/|x|^\gamma$ with $\gamma \in [1, 2)$ and $g:[0, \infty) \to (0, \infty)$ is a probability density with finite first moment satisfying $\sup_{t\geq 0} (1+t)g(t) <\infty$ and $\mathbf x_{s, t} \coloneqq \mathbf x_t - \mathbf x_s$ for $\mathbf x \in C([0, \infty), \R^d)$ and $s, t\geq 0$. By the results in \cite{BMPV22}, these conditions are sufficient for the existence of an infinite volume measure and the validity of a functional central limit theorem in infinite volume. 
By the proof of the central limit theorem given in \cite{MV19}\footnote{it is generalizable to the measures above, see Remark 4.14 in \cite{BP21}}, the respective diffusion constant $\sigma^2(\alpha)$ can also be obtained by taking  the ``diagonal limit'', that is
\begin{equation*}
	\sigma^2(\alpha) = \lim_{T \to \infty} \frac{1}{dT} \mathbb E_{\alpha, T}[|X_{0, T}|^2].
\end{equation*}
Here $X_t(\mathbf x) \coloneqq \mathbf x_t$ for any $\mathbf x \in C([0, \infty), \R^d)$ and $t\geq 0$.
	\begin{thrm}
		\label{Theorem: Our result}
		There exists a constant $C>0$ (depending on $\gamma$, $g$ and $d$) such that
		\begin{equation*}
		\sigma^2(\alpha) \leq C \alpha^{-2/(2+d)}
		\end{equation*}
		for all $\alpha>0$. Consequently, there exists a constant $c>0$ such that the effective mass of the Fröhlich Polaron satisfies $m_{\operatorname{eff}}(\alpha) \geq  c \alpha^{2/5}$ for all $\alpha>0$.
	\end{thrm}
	As we will point out in Remark \ref{Remark: Lowee bound effective mass}, the results obtained in \cite{BMPV22} imply for the Fröhlich Polaron with minor additional effort the existence of some $\tilde c>0$ such that $\sigma^2(\alpha) \geq e^{-\tilde c\alpha^2}$ for reasonably large $\alpha$.

	\section{Point process representation}
	\label{Section: point process representation}
	By using the identity
	$
		\frac{1}{|x|} = \sqrt{\frac{2}{\pi}} \int_0^\infty \mathrm du \, e^{-u^2 |x|^2/2}
	$
	and expanding the exponential into a series, Mukherjee and Varadhan \cite{MV19} represented the path measure of the Fröhlich Polaron as a mixture of Gaussian measures $\mathbf P_{\xi, u}$, the mixing measure $\widehat{\Theta}_{\alpha, T}$ being the distribution of a suitable point process on $\{(s, t)\in \R^2:\, s<t\} \times [0, \infty)$.
	In the following, we give an introduction to the point process representation in the form given in \cite{BP21}. We will additionally use the invariance of \eqref{Equation: Definition of path measures} under replacing $v$ by $v_\varepsilon(x) \coloneqq v(x) + \varepsilon$ for some $\varepsilon \geq 0$. While this transformation does not change the measure $\mathbb P_{\alpha, T}$, we will obtain different point process representations depending on the choice of $\varepsilon$.
	Let $\Gamma_{\alpha, T}$ be the distribution of a Poisson point process on $\triangle \coloneqq \{(s, t)\in \R^2:\, s<t\}$ with intensity measure
	\begin{equation*}
	\mu_{\alpha, T}(\mathrm ds \mathrm dt) \coloneqq \alpha g(t-s) \1_{\{0 \leq s<t\leq T\}}\mathrm ds \mathrm dt.
	\end{equation*}
	By expanding the exponential into a series and interchanging the order of summation and integration, we obtain for any $A\in \mathcal B(C([0, \infty), \R^d))$ and any $\varepsilon\geq 0$
	\begin{align}
	\label{Equation: Paths measure by PPP}
	&\mathbb P_{\alpha, T}(A) = \frac{1}{Z_{\alpha, T}^\varepsilon} \int_{A}\mathcal W(\mathrm d \mathbf x) \, \exp \left(\alpha \iint_{0\leq s<t\leq T}\mathrm ds \mathrm dt \, g(t-s) v_\varepsilon\big(\mathbf x_{s, t}\big)  \right) \nonumber \\ \nonumber
	&= \frac{1}{Z_{\alpha, T}^\varepsilon} \sum_{n=0}^\infty \frac{1}{n!} \int_{\triangle^n} \mu_{\alpha, T}^{\otimes n}(\mathrm d s_1 \mathrm d t_1, \hdots, \mathrm d s_n \mathrm d t_n) \int_A\mathcal W(\mathrm d\mathbf x) \prod_{i=1}^{n}v_\varepsilon\big(\mathbf x_{s_i, t_i}\big) \nonumber \\
	&= \frac{e^{c_{\alpha, T}}}{Z_{\alpha, T}^\varepsilon} \int_{\mathbf N_f(\triangle)} \Gamma_{\alpha, T}(\mathrm d\xi) \int_A \mathcal W( \mathrm d \mathbf x) \prod_{(s, t)\in \operatorname{supp}(\xi)  } v_\varepsilon\big(\mathbf x_{s, t}\big)
	\end{align}
	where $\mathbf N_f(\triangle)$ is the set of finite integer valued measures on $\triangle$. We will view $\xi = \sum_{i=1}^n \delta_{(s_i, t_i)}\in \mathbf N_f(\triangle)$ as a collection of intervals $\{[s_i, t_i]: 1\leq i \leq n\}$. The measure $\Gamma_{\alpha, T}$ can be interpreted in the following way: Consider a
	$M/G/\infty$-queue started empty at time $0$ where the arrival process $\sum_{n=0}^\infty \delta_{s_n}$ is a homogeneous Poisson process with rate $\alpha$, and 
	the service times $(\tau_n)_n$ are iid with density $g$ and are independent of the arrival process. Consider the process $\eta = \sum_{n=1}^\infty \delta_{(s_n, t_n)}$ where $t_n \coloneqq s_n + \tau_n$ is for $n\in \N$ the departure of customer $n$. Then the process off all customers in $\eta$ that arrive and depart before $T$ has distribution $\Gamma_{\alpha, T}$.	
 	For $\xi \in \mathbf N_f(\triangle)$ we define 
	\begin{equation*}
			F_\varepsilon(\xi) \coloneqq \mathbb E_\mathcal W\Big[\prod_{(s, t)\in \operatorname{supp}(\xi)  }v_\varepsilon(X_{s_i, t_i}) \Big]
	\end{equation*}
	and the perturbed point process $\widehat{\Gamma}_{\alpha, T}^\varepsilon$ by
	\begin{equation*}
		\widehat{\Gamma}_{\alpha, T}^\varepsilon(\mathrm d\xi) \coloneqq \frac{e^{c_{\alpha, T}}}{Z_{\alpha, T}^\varepsilon} F_\varepsilon(\xi) \Gamma_{\alpha, T}(\mathrm d\xi).
	\end{equation*}
	Then Equation \eqref{Equation: Paths measure by PPP} yields the representation  $\mathbb P_{\alpha, T}(\cdot) = \int \widehat{\Gamma}_{\alpha, T}^\varepsilon(\mathrm d \xi) \mathbf P^\varepsilon_\xi(\cdot)$ of $\mathbb P_{\alpha, T}$ as a mixture of the pertubed path measures 
	\begin{equation*}
	\mathbf P^\varepsilon_\xi(\mathrm d \mathbf x) \coloneqq \frac{1}{F_\varepsilon(\xi)} \prod_{(s, t)\in \operatorname{supp}(\xi)}v_\varepsilon(\mathbf x_{s_i, t_i}) \, \mathcal W(\mathrm d \mathbf x).
	\end{equation*}
	The function $F_\varepsilon$ depends in an intricate manner on the configuration $\xi$, depending on number, length and relative position of the intervals. Nevertheless, the perturbed measure $\widehat{\Gamma}^\varepsilon_{\alpha, T}$
can still be interpreted as a queuing process by identifying $(s,t) \in \operatorname{supp}(\xi)$ with a service 
interval $[s,t]$ as indicated above. 		
	Under the given assumptions on $g$ and $v$, $\widehat{\Gamma}^\varepsilon_{\alpha, T}$ can then be expressed in terms of an iid sequence of ``clusters'' of overlapping intervals, separated by exponentially distributed dormant periods not covered by any interval. The processes of increments can be drawn independently along these dormant periods and clusters (according to the kernel $(\xi, A) \mapsto \mathbf P^\varepsilon_\xi(A)$).	
	This yields the existence of the infinite volume limit as well as the functional central limit theorem \cite{BP21}, \cite{BMPV22}. In order to make the connection to the representation used in \cite{MV19}, we notice that for all $x\in \R^d \setminus \{0\}$
	\begin{align*}
	v_\varepsilon(x) = 1/|x|^\gamma + \varepsilon = \int_{[0, \infty)} \nu_\varepsilon(\mathrm du) e^{-u^2 |x|^2/2}
	\end{align*}
	where 
	\begin{equation*}
	\nu_\varepsilon(\mathrm du) = \frac{2^{(2-\gamma)/2}}{\Gamma(\gamma/2)} u^{\gamma-1} \mathrm du + \varepsilon \delta_0(\mathrm du).
	\end{equation*}
	It will turn out to be useful to make the atoms of out point process distinguishable.
	We will abuse notation and identify probability measures on $\mathbf N_f(\triangle)$ with symmetric probability measures on $\triangle^\cup \coloneqq \bigcup_{n=0}^\infty \triangle^n$.
	With Equation \eqref{Equation: Paths measure by PPP}, we obtain
	\begin{align*}
		\mathbb P_{\alpha, T}(A) = \frac{e^{c_{\alpha, T}}}{Z_{\alpha, T}^\varepsilon} \int_{\triangle^\cup} \Gamma_{\alpha, T}(\mathrm d\xi) \int_{[0, \infty)^{N(\xi)}} \nu_\varepsilon^{\otimes n}(\mathrm du) \int_A \mathcal W( \mathrm d \mathbf x) e^{-\sum_{i=1}^{N(\xi)} u_i^2 |\mathbf x_{s_i, t_i}|^2/2}
	\end{align*}
	where $N \coloneqq \sum_{n=0}^\infty n\cdot \1_{\triangle^n}$.
	We define for $\xi = ((s_i, t_i))_{1\leq i \leq n} \in \triangle^n$ and $u \in [0, \infty)^n$ the Gaussian measures 
	\begin{equation*}
		\mathbf P_{\xi, u}(\mathrm d\mathbf x) \coloneqq \frac{1}{\phi(\xi, u)} e^{-\sum_{i=1}^n u_i^2 |\mathbf x_{s_i, t_i}|^2/2} \, \mathcal W(\mathrm d \mathbf x) \\
	\end{equation*}
	where $\phi(\xi, u) \coloneqq \mathbb E_\mathcal W\big[e^{-\sum_{i=1}^n u_i^2 |X_{s_i, t_i}|^2/2}\big]$ and, for $\xi\in \triangle^n$ with\footnote{By finiteness of the partition function, we have $F_\varepsilon(\xi) < \infty$ for $\Gamma_{\alpha, T}$ almost all $\xi$} $F_\varepsilon(\xi) < \infty$ the measures
	\begin{equation*}
		\kappa_\varepsilon(\xi, \mathrm du) \coloneqq \frac{\phi(\xi, u)}{F_\varepsilon(\xi)} \nu_\varepsilon^{\otimes n}(\mathrm du)
	\end{equation*}
	and obtain
	\begin{equation}
	\label{Equation: Point process representation of measure}
		\mathbb P_{\alpha, T}(\cdot) = \int_{\triangle^\cup} \widehat{\Gamma}_{\alpha, T}^\varepsilon(\mathrm d\xi) \int_{[0, \infty)^{N(\xi)}} \kappa_\varepsilon(\xi, \mathrm du) \, \mathbf P_{\xi, u}(\cdot).
	\end{equation}
	We will view $(s, t, u) \in \triangle \times [0, \infty)$ as an interval $[s, t]$ equipped with the mark $u$. 
	In contrast to the representation in \cite{MV19}, we draw a sample of the mixing measure $\widehat{\Theta}_{\alpha, T}$ (which is the distribution of a point process on $\triangle \times [0, \infty)$) in two steps: First, draw a sample $\xi$ according to $\widehat{\Gamma}_{\alpha, T}^0$ and then draw marks $u$ according to the kernel $\kappa_0(\xi, \cdot)$. This added flexibility will be useful later.
	If we define 
	\begin{equation*}
		\sigma^2_T(\xi, u) \coloneqq  \frac{1}{dT}\mathbb E_{\mathbf P_{\xi, u}}[|X_{0, T}|^2]
	\end{equation*}
	we get with the previous considerations
	\begin{equation}
	\label{Equation: Point process representation of variance}
		\frac{1}{dT}\mathbb E_{\alpha, T}[|X_{0, T}|^2]= \int_{\triangle^\cup} \widehat{\Gamma}_{\alpha, T}^\varepsilon(\mathrm d\xi) \int_{[0, \infty)^{N(\xi)}} \kappa_\varepsilon(\xi, \mathrm du) \, \sigma^2_T(\xi, u).
	\end{equation}
	Taking the limit $T\to \infty$ in \eqref{Equation: Point process representation of variance} then yields the diffusion constant i.e. the inverse of the effective mass. One can also express the pertubed measure $\widehat{\Gamma}_{\alpha, T}^\varepsilon$ in terms of the path measures $\mathbb P_{\alpha, T}$ (a similar calculation was already made in \cite[(4.18)]{MV21}). The Laplace transform of $\widehat{\Gamma}^\varepsilon_{\alpha, T}$ evaluated at the measurable function $f: \triangle \to [0, \infty)$ is
	\begin{align*}
	\int_{\mathbf N_f(\triangle)} \widehat{\Gamma}_{\alpha, T}^\varepsilon(\mathrm d\xi) e^{-\int_\triangle f \mathrm d \xi}
	&= \frac{1}{Z_{\alpha, T}^\varepsilon} \sum_{n=0}^\infty \frac{1}{n!} \int_{\triangle^n} \mu_{\alpha, T}^{\otimes n}(\mathrm ds \mathrm dt) \mathbb E_\mathcal W\Big[ \prod_{i=1}^{n} v_\varepsilon(X_{s_i, t_i}) e^{-f(s_i, t_i)}\Big] \\
	&= \frac{1}{Z_{\alpha, T}^\varepsilon} \mathbb E_\mathcal W\bigg[ \sum_{n=0}^\infty \frac{1}{n!} \bigg(\int_{\triangle} \mu_{\alpha, T}(\mathrm ds \mathrm dt)  v_\varepsilon(X_{s, t}) e^{-f(s, t)} \bigg)^n \bigg] \\
	&= \frac{1}{Z_{\alpha, T}^\varepsilon} \mathbb E_\mathcal W\bigg[\exp\bigg(\int_{\triangle} \mu_{\alpha, T}(\mathrm ds \mathrm dt)  v_\varepsilon(X_{s, t}) e^{-f(s, t)} \bigg) \bigg] \\
	&= \mathbb E_{\alpha, T}\bigg[\exp\bigg(\int_{\triangle} \mu_{\alpha, T}(\mathrm ds \mathrm dt)  v_\varepsilon(X_{s, t}) \big(e^{-f(s, t)}-1\big) \bigg) \bigg].
	\end{align*}
	That is, $\widehat{\Gamma}_{\alpha, T}^\varepsilon$ is the distribution of a Cox process\footnote{I.e. $\widehat{\Gamma}^\varepsilon_{\alpha, T}$ is the distribution of a point process that is conditionally on $(X_t)_t \sim \mathbb P_{\alpha, T}$ a Poisson point process with intensity measure  $\mu_{\alpha, T}(\mathrm ds \mathrm dt) v_\varepsilon(X_{s, t})$} with driving measure $\mu_{\alpha, T}(\mathrm ds \mathrm dt) v_\varepsilon(X_{s, t})$ with 
	$(X_t)_{t} \sim \mathbb P_{\alpha, T}$. We introduced the parameter $\varepsilon$ in order to add a component to $\widehat{\Gamma}_{\alpha, T}^\varepsilon$ which is far easier to understand than the whole process: notice that we can write $\widehat{\Gamma}_{\alpha, T}^\varepsilon$ as the distribution of the sum $\eta_{\alpha, T} + \eta_{\alpha, T}^\varepsilon$ where $\eta_{\alpha, T} \sim \widehat{\Gamma}_{\alpha, T}^0$ and where the Poisson point process $\eta_{\alpha, T}^\varepsilon \sim \Gamma_{\varepsilon \alpha, T}$ is independent of $\eta_{\alpha, T}$. Notice also that the distribution of marks of $\eta_{\alpha, T}^\varepsilon$ under $\kappa_\varepsilon$ still depends on the whole process.
	We will prove Theorem \ref{Theorem: Our result} in three steps:
	\begin{enumerate}
		\item We identify configurations $(\xi, u)$ for which $\sigma_T^2(\xi, u)$ is small. We will see that $\sigma^2_T(\xi, u)$ is small if there exists some subconfiguration $(\xi', u')$ of $(\xi, u)$ (obtained by deletion of marked intervals) such that $\xi'$ consists of pairwise disjoint intervals which cover large parts of $[0, T]$ and have lengths exceeding 1 and marks exceeding some large $C(\alpha)$. While not necessary\footnote{Lemma \ref{Lemma: sigmaT for partition} later in the text can also be shown directly by using independence of Brownian increments} for the proof of Theorem \ref{Theorem: Our result}, we give a characterization of $\sigma^2_T(\xi,u)$ in terms of a $L^2$-distance which emphasizes the variational type flavor of our approach and might be useful for further refining the method.
		\item For some $C(\alpha)$ (to be determined later), we estimate the kernel $\kappa_\varepsilon$ by a product kernel that marks intervals independent of each other with a mark in $\{0, C(\alpha)\}$. In particular, the product kernel allows us to mark $\eta_{\alpha, T}^\varepsilon$ without knowledge of the full process $\eta_{\alpha, T} + \eta_{\alpha, T}^\varepsilon$.
		\item We mark $\eta_{\alpha, T}^\varepsilon$ with this product kernel and obtain a Poisson point process on $\triangle \times [0, \infty)$. Thinning this process by only considering all marked intervals whose length exceeds 1 and that have the mark $C(\alpha) = \alpha^{1/(2+d)}$ yields Theorem \ref{Theorem: Our result}.
	\end{enumerate}
	
	\begin{remark}
		\label{Remark: Lowee bound effective mass}
	For the Fröhlich Polaron, i.e. for $d=3$, $\gamma =1$ and $g(t) = e^{-t}$ for all $t\geq 0$, the point process representation implies the existence of some constant $\tilde c>0$ such that $\sigma^2(\alpha) \geq e^{-\tilde c \alpha^2}$ for all reasonably large $\alpha>0$. Given $\xi \in \mathbf N_f(\triangle)$ and $r\in \R$ let $N_r(\xi) \coloneqq \xi(\{(s,t): s\leq r \leq t\})$ be the number of intervals that contain $r$. We call $t$ dormant under $\xi$, if $N_t(\xi) = 0$, else we call $t$ active under $\xi$. Under $\xi$, the interval $[0, T]$ can be decomposed into alternating dormant and active periods. Let
		\begin{equation*}
		\mathcal D_T(\xi) \coloneqq \lambda(\{t \in [0, T]: N_t(\xi) = 0\})
		\end{equation*}
		(where $\lambda$ is the Lebesgue measure) be the sum of lengths of all dormant periods in $[0, T]$. One can easily convince oneself that
		$
		\sigma^2_T(\xi, u) \geq \frac{1}{T}\mathcal D_T(\xi)
		$
		for all $u\in [0, \infty)^n$  (this follows from Proposition \ref{Proposition: Point process representation of diffusion constant} later in the text as well). It was shown in \cite{BMPV22} that
		\begin{equation*}
		\liminf_{T \to \infty} \frac{1}{T}\int_0^T \widehat \Gamma_{\alpha, T}^0(N_t = 0) \, \mathrm dt =  \liminf_{T \to \infty} \frac{1}{T}\int_0^T \frac{Z_{\alpha, t} Z_{\alpha, T-t}}{Z_{\alpha, T}} \, \mathrm dt \geq e^{-\tfrac{3}{2} \alpha \psi'(\alpha)}
		\end{equation*}
		where $\psi(\alpha) \coloneqq \lim_{T\to \infty} \log(Z_{\alpha, T})/T= - E_\alpha(0)$ is minus the ground state energy of the total momentum zero Hamiltonian. By Fubinis Theorem
		\begin{equation*}
		\int_0^T \widehat \Gamma_{\alpha, T}^0(N_t = 0) \, \mathrm dt =  \mathbb E_{\widehat \Gamma^0_{\alpha, T}}[\mathcal D_T] 
		\end{equation*}
		and thus $\sigma^2(\alpha) \geq e^{-\tfrac{3}{2} \alpha \psi'(\alpha)}$ by \eqref{Equation: Point process representation of variance}. By convexity of $\psi$ and since $\lim_{\alpha \to \infty} \psi(\alpha)/\alpha^2 \in (0, \infty)$ we have
		\begin{equation*}
			\psi'(\alpha) \leq \frac{\psi(2\alpha) - \psi(\alpha)}{\alpha} \in \mathcal O(\alpha)
		\end{equation*}
		 and hence there exists some $\tilde c>0$, $\alpha_0 >0$ such that $\sigma^2(\alpha) \geq e^{-\tilde c\alpha^2}$ for all $\alpha \geq \alpha_0$.
	\end{remark}
	
	\section{Properties of $\sigma_T^2(\xi, u)$}
	We derive a few properties of $\sigma_T^2(\xi, u)$. First, we will show that $\sigma_T^2(\xi, \cdot)$ is decreasing with respect to the partial order on $[0, \infty)^{N(\xi)}$ defined by $u\leq \tilde u$ iff $u_i\leq \tilde u_i$ for all $1\leq i \leq N(\xi)$. We will then express $\sigma_T^2(\xi, u)$ in terms of a $L^2$-distance and calculate $\sigma_T^2(\xi, u)$ in case that $\xi$ consists of pairwise disjoint intervals.
		\begin{lemma}
		\label{Lemma: sigmaT decreasign in u}
		For fixed $\xi \in \triangle^\cup$, the function $\sigma^2_T(\xi, \cdot)$ is decreasing on $[0, \infty)^{N(\xi)}$.
	\end{lemma}
	\begin{proof}
		Let $ \xi = ((s_i, t_i))_{1\leq i \leq n} \in \triangle^n$. Then
		\begin{align*}
			\frac{\partial \sigma^2_T(\xi, u)}{\partial u_i} &= \frac{\partial }{\partial u_i} \frac{\mathbb E_\mathcal W\big[|X_{0, T}|^2 e^{-\sum_{j=1}^{n} u_j^2 |X_{s_j, t_j}|^2/2}\big]}{\mathbb E_\mathcal W\big[e^{-\sum_{j=1}^{n} u_j^2 |X_{s_j, t_j}|^2/2}\big]} \\
			&=-u_i \operatorname{Cov}_{\mathbf P_{\xi, u}}(|X_{0, T}|^2,|X_{s_i, t_i}|^2).
		\end{align*}
		By Isserlis' theorem (and as the coordinate processes $X^1, \hdots, X^d$ are iid under the centered Gaussian measure $\mathbf P_{\xi, u}$)
		\begin{align*}
			\operatorname{Cov}_{\mathbf P_{\xi, u}}(|X_{0, T}|^2,|X_{s_i, t_i}|^2) &= d \cdot \operatorname{Cov}_{\mathbf P_{\xi, u}}((X_{0, T}^1)^2,(X_{s_i, t_i}^1)^2) \\
			&= 2d\cdot\mathbb E_{\mathbf P_{\xi, u}}\big[X_{0, T}^1 X_{s_i, t_i}^1 \big]^2 \geq 0. \qedhere
		\end{align*}
	\end{proof}
	In particular, $\sigma^2_{T}(\xi, u)$ is increasing under deletion of a marked interval $(s_i, t_i, u_i)$ (set $u_i = 0$).
	As already used in \cite{MV19}, the normalization constant $\phi(\xi, u)$ can also be expressed in terms of a determinant of the covariance matrix of a suitable Gaussian vector.
	Let $(B_t)_{t\geq 0}$ be an one dimensional Brownian motion and $(Z_n)_n$ be an iid sequence of $\mathcal N(0, 1)$ distributed random variables, independent of $(B_t)_{t\geq 0}$. Then, by Lemma \ref{Lemma: phi as a determiant} of the appendix, we have
	\begin{equation*}
	\phi(\xi, u) = \frac{1}{\det( C(\xi, u))^{d/2}}
	\end{equation*}
	where $C(\xi, u)$ is the covariance matrix of the Gaussian vector $(u_1 B_{s_1, t_1} + Z_1, \hdots, u_n B_{s_n, t_n} + Z_n)$.
	\begin{prop}
		\label{Proposition: Point process representation of diffusion constant}
		The diffusion constant has the representation
		\begin{equation*}
		\sigma^2(\alpha)= \lim_{T \to \infty} \int_{\triangle^\cup} \widehat{\Gamma}_{\alpha, T}^\varepsilon(\mathrm d\xi) \int_{[0, \infty)^{N(\xi)}} \kappa_\varepsilon(\xi, \mathrm du) \, \sigma^2_T(\xi, u)
		\end{equation*}
		where for $\xi = ((s_i, t_i))_{1\leq i \leq n} \in \triangle^n$ and $u\in [0, \infty)^n$
		\begin{equation}
		\label{Equation: Variance in terms of L2 distance}
			\sigma^2_T(\xi, u)
			= \frac{1}{T}\operatorname{dist}_{L^2}\Big(B_{0, T}, \operatorname{span}\{u_i B_{s_i, t_i} + Z_i: \, 1\leq i \leq n\}\Big)^2.
		\end{equation}
	\end{prop}
		\begin{proof}
		It is left to show Equality \eqref{Equation: Variance in terms of L2 distance}. We have
		\begin{equation}
		\label{Equation: sigma2(xi, u) as derivative}
		T \sigma^2_T(\xi, u) = -\frac{2/d}{\phi(\xi, u)}\frac{\partial}{ \partial v } \mathbb E_\mathcal W\Big[ e^{- \tfrac{1}{2} v |X_{0, T}|^2- \sum_{i=1}^n \tfrac{1}{2} u_i^2 |X_{s_i, t_i}|^2 } \Big] \Big|_{v=0}.
		\end{equation}
		For $v, w \in \R$ let $C_T(\xi, u, v, w) \in \R^{(n+1)\times (n+1)}$ be the covariance matrix of the Gaussian vector $(u_1 B_{s_1, t_1} + Z_1, \hdots, u_n B_{s_n, t_n} + Z_n, \sqrt{v}B_{0, T} + w Z_{n+1})$. Then, with Lemma \ref{Lemma: phi as a determiant}, Equation \eqref{Equation: sigma2(xi, u) as derivative} becomes 
		\begin{align*}	
		T \sigma_T^2(\xi, u) &= - \frac{2}{d} \det(C(\xi, u))^{d/2} \frac{\partial}{ \partial v }\frac{1}{\det(C_T(\xi, u, v, 1))^{d/2}}\Big|_{v=0} \\
		&= \frac{1}{\det(C(\xi, u))} \frac{\partial}{ \partial v } \det(C_T(\xi, u, v, 1)) \Big|_{v=0}
		\end{align*}
		since $\det(C_T(\xi, u, 0, 1)) = \det(C(\xi, u))$. By using that the determinant is multilinear in the rows and columns, we have
		\begin{equation*}
			\det(C_T(\xi, u, v, 1)) = \det(C(\xi, u)) + v\det(C_T(\xi, u, 1, 0))
		\end{equation*}
		and hence
		\begin{equation*}
			\frac{\partial}{ \partial v } \det(C_T(\xi, u, v, 1)) \Big|_{v=0} = \det(C_T(\xi, u, 1, 0)).
		\end{equation*}
		Finally, by Lemma \ref{Lemma: Determiant in terms of L2 distance}
		\begin{equation*}
			\frac{\det(C_T(\xi, u, 1, 0))}{\det(C(\xi, u))} = \operatorname{dist}_{L^2}\Big(B_{0, T},  \operatorname{span}\{u_i B_{s_i, t_i} + Z_i: \, 1\leq i \leq n\}\Big)^2
		\end{equation*}
		which concludes the proof.		
	\end{proof}

	\begin{lemma}
		\label{Lemma: sigmaT for partition}
		Let $\xi = ((s_i, t_i))_{1\leq i \leq n} \in \triangle^n$ with $[s_i, t_i] \subseteq [0, T]$ for all $1\leq i \leq n$ and $[s_i, t_i] \cap [s_j, t_j] = \emptyset$ for $i\neq j$. Then, for $u \in [0, \infty)^n$
		\begin{equation*}
			\sigma_T^2(\xi, u) =  1- \frac{1}{T} \sum_{i=1}^n \tau_i + \frac{1}{T}\sum_{i=1}^n \frac{\tau_i}{\tau_i u_i^2 + 1}
		\end{equation*}
		where $\tau_i \coloneqq t_i - s_i$ for $1\leq i \leq n$.
	\end{lemma}
	
		\begin{proof}
		We project $B_{0, T}$ onto $\operatorname{span} \{u_iB_{s_i, t_i} + Z_i:\, 1 \leq i \leq n\}$ and obtain
		\begin{align*}
		\sum_{i=1}^n \mathbb E[B_{0, T}\cdot (u_i B_{s_i, t_i} + Z_i)] \frac{u_i B_{s_i, t_i} + Z_i}{\mathbb E[(u_i B_{s_i, t_i} + Z_i)^2]} = \sum_{i=1}^n  \frac{\tau_iu_i}{u_i^2\tau_i + 1} (u_i B_{s_i, t_i} + Z_i).
		\end{align*}
		Hence
		\begin{align*}
		T \sigma^2_T(\xi, u) &=\operatorname{dist}\big(B_{0, T}, \operatorname{span} \{u_iB_{s_i, t_i} + Z_i:\, 1\leq i \leq n\} \big)^2 \\
		&= \mathbb E\Big[\Big(B_{0, T} - \sum_{i=1}^n  B_{s_i, t_i} + \sum_{i=1}^n  \frac{1}{u_i^2\tau_i + 1} B_{s_i, t_i} - \frac{u_i\tau_i}{u_i^2\tau_i + 1}Z_i \Big)^2 \Big] \\
		&= \mathbb E\Big[\Big(B_{0, T} - \sum_{i=1}^n  B_{s_i, t_i} \Big)^2 \Big]  + \sum_{i=1}^n  \frac{\tau_i}{(u_i^2\tau_i + 1)^2} + \frac{u_i^2 \tau_i^2}{(u_i^2\tau_i + 1)^2} \\
		&= T - \sum_{i=1}^n \tau_i + \sum_{i=1}^n  \frac{\tau_i}{\tau_i u_i^2 + 1}. \qedhere
		\end{align*}
	\end{proof}
	
	In the situation above, $1 - \frac{1}{T}\sum_{i=1}^n \tau_i = \frac{1}{T} \lambda\big([0, T]\setminus\bigcup_{i=1}^n[s_i, t_i]\big)$, where $\lambda$ denotes the Lebesgue measure.
	Assume that $\tau_i \geq 1$ and $u_i \geq C$ for all $1\leq i \leq n$ where $C>0$. Then
	\begin{equation}
	\label{Equation: Second errorterm}
		 \frac{1}{T}\sum_{i=1}^n \frac{\tau_i}{\tau_i u_i^2 + 1} \leq \frac{1}{C^2 + 1} \frac{1}{T}\sum_{i=1}^n \tau_i \leq \frac{1}{C^2 + 1}.
	\end{equation}
	For a general configuration $(\xi, u)$, the previous considerations allow us to obtain estimates on $\sigma_{T}^2(\xi, u)$ by considering subconfigurations $(\xi', u')$ of pairwise disjoint marked intervals with intervals lengths exceeding  $1$ and marks exceeding $C$.
	In combination with an application of renewal theory, this implies the following technical Lemma:
	
	\begin{lemma}
		\label{Lemma: Thinned queue covers interval}
		Let $f:[0, \infty) \to [0, \infty)$ be measurable with $\int_0^\infty (1+t)f(t) \, \mathrm dt <\infty$ and $f(t) = 0$ for almost all $t\in [0, 1]$.
		For $T>0$ let $\Xi_T$ be the law of a Poisson point process with intensity measure 
		$
		f(t-s) \1_{\{0 < s <t<T\}} \mathrm ds \mathrm dt
		$ and let $C>0$. Then
		\begin{equation*}
		\limsup_{T \to \infty} \int_{\triangle^\cup} \Xi_T(\mathrm d \xi) \int_{[0, \infty)^{N(\xi)}} \delta_C^{\otimes N(\xi)}(\mathrm du)\, \sigma^2_T(\xi, u) \leq \frac{1}{1+ \int_0^\infty t f(t) \, \mathrm dt} + \frac{1}{C^2 + 1}.
		\end{equation*}
	\end{lemma}
	\begin{proof}
		In case that $f=0$ a.e. the statement is trivial, so assume $\int_0^\infty f(t) \, \mathrm dt >0$. Let
		\begin{equation*}
		\beta \coloneqq \int_{0}^{\infty} f(r) \mathrm dr, \quad \rho(t) \coloneqq \frac{f(t)}{\int_{0}^{\infty} f(r) \mathrm dr}
		\end{equation*}
		for $t\geq 0$. We consider the Poisson point process $\eta = \sum_{n=1}^\infty \delta_{(s_n, t_n)}$ of a $M/G/\infty$-queue where the arrival process $\sum_{n=1}^\infty \delta_{s_n}$ is a homogeous Poisson point process with intensity $\beta$, the service times $(\tau_n)_n$ are iid, have density $\rho$ and are independent of the arrival process and $t_n \coloneqq s_n + \tau_n$ is for $n\in \N$ the departure of customer $n$.
		Then the restriction of $\eta$ to the process of all customers that arrive and depart before $T$ has distribution $\Xi_T$. We inductively define $i_0 \coloneqq 0$, $t_0 \coloneqq 0$ and
		\begin{equation*}
		i_{n+1} \coloneqq \inf \{j > i_n: s_{j} > t_{i_n}\}
		\end{equation*}
		for $n\geq 0$, i.e. customer $i_{n+1}$ is for $n\geq 1$ the first customer that arrives after the departure of customer $i_n$. The waiting times $(s_{i_{n}} - t_{i_{n-1}})_{n\geq 1}$ are iid $\operatorname{Exp}(\beta)$ distributed and are independent of the iid service times $(t_{i_n} - s_{i_n})_{n\geq 1}$ which have density $\rho$. Notice that $(t_{i_n})_{n\geq 0}$ defines a renewal process. Let $N_T \coloneqq \inf\{n: t_{i_n} \geq T\}$. By the considerations after Lemma \ref{Lemma: sigmaT for partition}, we have
		\begin{equation*}
		\int_{\triangle^\cup} \Xi_{T}(\mathrm d \xi) \int_{[0, \infty)^{N(\xi)}} \delta_C^{\otimes N(\xi)}(\mathrm du)\, \sigma^2_T(\xi, u) \leq \mathbb E\Big[\frac{B_T}{T} + \frac{1}{T}\sum_{n=1}^{N_T-1} s_{i_n} - t_{i_{n-1}} \Big] + \frac{1}{C^2 +1}
		\end{equation*}
		where $t_0 \coloneqq 0$ and $B_T \coloneqq T - t_{i_{N_T-1}}$ is the backward recurrence time of the renewal process $(t_{i_n})_{n}$ at time $T$. By renewal theory, $(B_T)_{T\geq 0}$ converges in distribution as $T\to \infty$. Let $(T_k)_k$ be a sequence in $[0, \infty)$ that converges to infinity. Then $B_{T_k}/T_k \to 0$ in probability as $k\to \infty$. Hence, there exists a subsequence $(T_{k_j})_j$ such that $B_{T_{k_j}}/T_{k_j} \to 0$ almost surely as $j\to \infty$. By dominated convergence, $\mathbb E[B_{T_{k_j}}/T_{k_j}] \to 0$ as $j\to \infty$. As $(T_k)_k$ was arbitrary, we have $\mathbb E[B_{T}/T] \to 0$ as $T\to\infty$. By renewal theory
		\begin{equation}
		\label{Equation: Number of renewal points}
		N_T/T \to \frac{1}{\mathbb E[t_{i_1}]} = \frac{1}{\beta^{-1} + \int_0^\infty r \rho(r) \, \mathrm dr}
		\end{equation}
		almost surely as $T\to \infty$ and 	
		by the law of large numbers we thus have
		\begin{equation*}
		\frac{1}{T}\sum_{n=1}^{N_T-1} s_{i_n} - t_{i_{n-1}} \to  \frac{\beta^{-1}}{\beta^{-1} + \int_0^\infty r \rho(r) \, \mathrm dr} = \frac{1}{1 + \int_0^\infty r f(r) \, \mathrm dr}
		\end{equation*}
		almost surely as $T\to \infty$. By the dominated convergence theorem, we have convergence in $L^1$ as well and the statement follows.
	\end{proof}
	
	\section{Estimation of the kernel $\kappa_\varepsilon$}
	While the component $\eta_{\alpha, T}^\varepsilon \sim \Gamma_{\varepsilon\alpha, T}$ is far easier to understand than the whole process, the distribution of marks of $\eta_{\alpha, T}^\varepsilon$ under $\kappa_\varepsilon$ still depends on the whole process. We get rid of this problem by estimating $\kappa_\varepsilon$ by a suitable product kernel that marks the intervals independent of each other. For $t, C>0$ we define
	\begin{equation*}
	h_\varepsilon(t) \coloneqq \mathbb E_\mathcal W[v_\varepsilon(X_t)] = \int_{[0, \infty)} \nu_\varepsilon(\mathrm du) \frac{1}{(1+u^2t)^{d/2}}
	\end{equation*}
	and
	\begin{equation*}
	\label{Equation: Definiton p(C)}
	p_\varepsilon(C) \coloneqq \frac{2^{-\gamma/2} }{(1+4C^2)^{d/2}}\big(1- \varepsilon/h_\varepsilon(2)\big).
	\end{equation*}
	Notice that there exists some constant $c_{\gamma, d}$ such that $h_\varepsilon(t) = c_{\gamma, d} t^{-\gamma/2} + \varepsilon$ for all $t>0$. In particular, $\varepsilon/h_\varepsilon(2)<1$ and
	\begin{equation*}
	\mathcal P_{\varepsilon, C}(t, \cdot ) \coloneqq
	\begin{cases}
	\delta_0 \quad &\text{ if } t> 2 \\
	p_\varepsilon(C) \delta_C + (1-p_\varepsilon(C))\delta_0 &\text { if } t\leq 2
	\end{cases}
	\end{equation*}
	defines for $t>0$ a probability measure. Fix $\xi = ((s_i, t_i))_{1\leq i\leq N(\xi)} \in \triangle^\cup$ with $F_\varepsilon(\xi)< \infty$.
	We will show that $\kappa_\varepsilon(\xi, \cdot)$ stochastically dominates
	\begin{equation*}
	\tilde \kappa_{\varepsilon, C}(\xi, \cdot) \coloneqq \bigotimes_{i=1}^{N(\xi)} \mathcal P_{\varepsilon, C}(t_i-s_i, \, \cdot \,)
	\end{equation*}
	that is
	\begin{equation*}
		\int_{[0, \infty)^{N(\xi)}} \kappa_\varepsilon(\xi, \mathrm du) f(u) \leq \int_{[0, \infty)^{N(\xi)}} \tilde \kappa_{\varepsilon, C}(\xi, \mathrm du) f(u)
	\end{equation*}
	for any decreasing function $f: [0, \infty)^{N(\xi)} \to \R$.
	By Proposition \ref{Proposition: Point process representation of diffusion constant} and the monotonicity of $\sigma_T^2(\xi, \cdot)$, we then obtain
	\begin{equation*}
	\sigma^2(\alpha) \leq \limsup_{T \to \infty} \int_{\triangle^\cup} \widehat{\Gamma}_{\alpha, T}^\varepsilon(\mathrm d\xi) \int_{\{0, C\}^{N(\xi)}} \tilde \kappa_{\varepsilon, C}(\xi, \mathrm du) \, \sigma_{T}^2(\xi, u).
	\end{equation*}
	Since $\sigma_{T}^2(\xi, u)$ is increasing under deletion of marked intervals we can further estimate
	\begin{equation}
	\label{Equation: Estimate diffusion constant with PPP}
	\sigma^2(\alpha) \leq \limsup_{T \to \infty} \int_{\triangle^\cup} \Gamma_{\varepsilon \alpha, T}(\mathrm d\xi) \int_{\{0, C\}^{N(\xi)}} \tilde \kappa_{\varepsilon, C}(\xi, \mathrm du) \, \sigma_{T}^2(\xi, u).
	\end{equation}
	Notice that intervals with zero mark do not contribute to $\sigma^2_T(\xi, u)$. Estimate \eqref{Equation: Estimate diffusion constant with PPP} in combination with Lemma \ref{Lemma: Thinned queue covers interval}
	will yield Theorem \ref{Theorem: Our result}.
	
	\begin{lemma}
		\label{Lemma: stochastic domination of chi}
		Let $C>0$ and $\xi = ((s_i, t_i))_{1 \leq i \leq n} \in \triangle^\cup$ with $F_\varepsilon(\xi)<\infty$. Then $\kappa_\varepsilon(\xi,\, \cdot \,)$ stochastically dominates $\tilde \kappa_{\varepsilon, C}(\xi, \cdot)$.
	\end{lemma}
	\begin{proof}
		Let
	    $
			J(\xi) \coloneqq \{1\leq i \leq n:\, t_i - s_i \leq 2\}
		$.
		The statement is trivial for $J(\xi) = \emptyset$, so assume $J(\xi) \neq \emptyset$.
We define
\begin{equation*}
	\chi:[0, \infty)^n \to \{0, C\}^{J(\xi)}, \quad \chi_i \coloneqq C \cdot \1_{\{u_i\geq C\}}
\end{equation*}
for $i \in J(\xi)$. Let $\pi: [0, \infty)^{n} \to [0, \infty)^{J(\xi)}$ be the restriction map $u \mapsto (u_i)_{i \in J(\xi)}$ and let $\phi: [0, \infty)^{n} \to [0, \infty)^n$ be defined by
\begin{equation*}
	\phi_i(u) \coloneqq \begin{cases}
	u_i &\text{ if }i\in J(\xi) \\
	0   &\text{ else}
	\end{cases}
\end{equation*}
for $u \in [0, \infty)^n$. Let $f: [0, \infty)^{n} \to \R$ be decreasing and let $\tilde f:[0, \infty)^{J(\xi)} \to \R$ be such that $\tilde f \circ \pi = f \circ \phi$. Then
\begin{align*}
	\int_{[0, \infty)^{n}} \kappa_\varepsilon(\xi, \mathrm du) f(u) &\leq \int_{[0, \infty)^{n}} \kappa_\varepsilon(\xi, \mathrm du) \tilde f(\chi(u)) \\
	\int_{[0, \infty)^{n}} \tilde \kappa_{\varepsilon, C}(\xi, \mathrm du) f(u) &= \int_{[0, \infty)^{n}} \tilde \kappa_{\varepsilon, C}(\xi, \mathrm du) \tilde f(\chi(u)).
\end{align*}
		Hence, it is sufficient to show that the distribution of $\chi$ under $\mu \coloneqq \kappa_\varepsilon(\xi, \cdot)$ stochastically dominates the distribution of $\chi$ under $\tilde \mu \coloneqq \tilde \kappa_{\varepsilon, C}(\xi, \cdot)$.
		In order to show this, it is sufficient (see e.g. \cite[pp. 137-138]{FV17}) to show that for any $\omega \in \{0, C\}^{J(\xi)}$ and any $i \in J(\xi)$
		\begin{equation*}
			\mu(\chi_i = C| \chi_j = \omega_j \forall j \neq i) \geq p_{\varepsilon}(C) = \tilde \mu(\chi_i = C| \chi_j = \omega_j \forall j \neq i).
		\end{equation*}
		To simplify notation, we assume w.l.o.g. that $i=1 \in J(\xi)$ and set $\tau_1 \coloneqq t_1-s_1$.
		Let $\omega \in \{0, C\}^{J(\xi)}$ be fixed and let $A \subseteq [0, \infty)^{n-1}$ be such that
		$
		 \{\chi_j = \omega_j \, \forall j \neq 1\} = [0, \infty) \times A.
		$	
		For $u_1\in [0, \infty)$ and $u = (u_2, \hdots, u_{n}) \in [0, \infty)^{n-1}$ we have by Corollary \ref{Corollary: Factor out of phi} of the appendix and since $(C+u_1)^2/2 \leq C^2 + u_1^2$
		\begin{align*}
			\phi(\xi, (u_1+C, u)) &= \mathbb E_\mathcal W \Big[e^{-(C + u_1)^2 |X_{s_1, t_1}|^2/2}  \prod_{i=2}^{n }e^{-u_i^2 |X_{s_i, t_i}|^2/2} \Big] \\
			&\geq \mathbb E_\mathcal W \Big[e^{-C^2 |X_{s_1, t_1}|^2} e^{-u_1^2|X_{s_1, t_1}|^2}  \prod_{i=2}^{n }e^{-u_i^2 |X_{s_i, t_i}|^2/2} \Big] \\
			&\geq  \frac{1}{(1+ 4C^2)^{d/2}} \phi(\xi, (\sqrt{2}u_1, u))
		\end{align*}
		since $\tau_1 \leq 2$ as $1\in J(\xi)$. We get (by using the fact that the density of $\nu_\varepsilon|_{(0, \infty)}$ with respect to the Lebesgue measure is increasing)
		\begin{align*}
			&\mu\big(\chi_1 = C, \chi_j = \omega_j \, \forall j\neq 1\big) \\
			&= \frac{1}{F_\varepsilon(\xi)}\int_{(C, \infty)}\nu_\varepsilon(\mathrm du_1) \int_{A} \nu_\varepsilon^{\otimes(n-1)}(\mathrm du) \phi(\xi, (u_1, u)) \\
			&\geq \frac{1}{F_\varepsilon(\xi)}\int_{(0, \infty)}\nu_\varepsilon(\mathrm du_1) \int_{A} \nu_\varepsilon^{\otimes(n-1)}(\mathrm du) \phi(\xi, (u_1+C, u)) \\
			&\geq \frac{1}{(1+ 4C^2)^{d/2}} \frac{1}{F_\varepsilon(\xi)}\int_{(0, \infty)}\nu_\varepsilon(\mathrm du_1) \int_{A} \nu_\varepsilon^{\otimes(n-1)}(\mathrm du) \phi(\xi, (\sqrt{2}u_1, u)) \\
			&= \frac{2^{-\gamma/2} }{(1+ 4C^2)^{d/2}} \frac{1}{F_\varepsilon(\xi)}\int_{(0, \infty)}\nu_\varepsilon(\mathrm du_1) \int_{A} \nu_\varepsilon^{\otimes(n-1)}(\mathrm du) \phi(\xi, (u_1, u)) \\
			&=\frac{2^{-\gamma/2} }{(1+ 4C^2)^{d/2}} \frac{1}{F_\varepsilon(\xi)}\bigg(\int_{[0, \infty)}\nu_\varepsilon(\mathrm du_1)  \int_{A} \nu_\varepsilon^{\otimes(n-1)}(\mathrm du) \phi(\xi, (u_1, u))\\
			& \quad \quad \quad \quad \quad \quad \quad \quad \quad \quad \quad \quad - \varepsilon \int_{A} \nu_\varepsilon^{\otimes(n-1)}(\mathrm du) \phi(\xi', u) \bigg)
		\end{align*}
		where $\xi' = ((s_2, t_2), \hdots, (s_n, t_n))$. By Corollary \ref{Corollary: Factor out of phi}, for any $u_1\in [0, \infty), u\in [0, \infty)^{n-1}$
		\begin{equation*}
			\phi(\xi, (u_1, u)) \geq \frac{ \phi(\xi', u)}{(1+u_1^2 \tau_1)^{d/2}} \geq \frac{ \phi(\xi', u)}{(1+2u_1^2)^{d/2}} 
		\end{equation*}
		which yields
		\begin{align*}
			\int_{[0, \infty)}\nu_\varepsilon(\mathrm du_1)  \int_{A} \nu_\varepsilon^{\otimes(n-1)}(\mathrm du) \phi(\xi, (u_1, u))
			&\geq \int_{[0, \infty)} \frac{\nu_\varepsilon(\mathrm du_1)}{(1+2u_1^2)^{d/2}} \int_{A} \nu_\varepsilon^{\otimes(n-1)}(\mathrm du) \phi(\xi', u) \\
			&= h_\varepsilon(2) \int_{A} \nu_\varepsilon^{\otimes(n-1)}(\mathrm du) \phi(\xi', u).
		\end{align*}
		Hence, we get
		\begin{align*}
		&\mu\big(\chi_1 = C, \chi_j = \omega_j \,\forall j\neq 1\big) \\
		&\geq  \frac{2^{-\gamma/2}}{(1+ 4C^2)^{d/2}} \Big(1- \varepsilon/h_\varepsilon(2) \Big) \frac{1}{F_\varepsilon(\xi)} \int_{[0, \infty)}\nu_\varepsilon(\mathrm du_1)  \int_{A} \nu_\varepsilon^{\otimes(n-1)}(\mathrm du) \phi(\xi, (u_1, u)) \\
		&= p_\varepsilon(C) \mu \big(\chi_j = \omega_j \, \, \forall j\neq 1\big)
		\end{align*}
		which yields the claim. 
	\end{proof}

	\section{Finishing the proof of Theorem \ref{Theorem: Our result}}
	
	\begin{proof}[Proof of Theorem \ref{Theorem: Our result}]
	For $C>0$ let $\zeta_{\alpha, T}^{\varepsilon, C}$ be a Poisson point process on $\triangle \times \{0, C\}$ with intensity measure
	\begin{equation*}
	\alpha \varepsilon g(t-s) \1_{\{0 < s< t < T\}} \mathrm ds \mathrm dt \, \mathcal P_{\varepsilon, C}(t-s, \mathrm du).
	\end{equation*}
	Then Estimate \eqref{Equation: Estimate diffusion constant with PPP} becomes\footnote{Here $\sigma_T^2((s_1, t_1, u_1), \hdots, (s_n, t_n, u_n)) \coloneqq \sigma_T^2\big(((s_1, t_1),\hdots, (s_n, t_n)), (u_1, \hdots, u_n) \big)$}
	\begin{equation}
	\label{Equation: Estimate disffusion constant PPP}
	\sigma^2(\alpha) \leq \limsup_{T \to \infty} \mathbb E[\sigma_T^2(\zeta^{\varepsilon, C}_{\alpha, T})] \leq \limsup_{T \to \infty} \mathbb E[\sigma_T^2(\tilde \zeta^{\varepsilon, C}_{\alpha, T})]
	\end{equation}
	where $\tilde \zeta_{\alpha, T}^{\varepsilon, C}$ is the restriction of $\zeta_{\alpha, T}^{\varepsilon, C}$ to $\{(s, t): 1 \leq t-s \leq 2\} \times \{C\}$ and thus a Poisson point process with intensity measure
	\begin{equation*}
	\alpha \varepsilon p_\varepsilon(C) g(t-s) \1_{\{1\leq t-s \leq 2\}} \1_{\{0 < s< t < T\}} \mathrm ds \mathrm dt \mathcal \, \delta_{C}(\mathrm du).
	\end{equation*}
	By Lemma \ref{Lemma: Thinned queue covers interval} we have
	\begin{equation*}
	\limsup_{T \to \infty} \mathbb E[\sigma_T^2(\tilde \zeta^{\varepsilon, C}_{\alpha, T})] \leq \frac{1}{1 + \alpha \varepsilon p_\varepsilon(C) \int_1^2 r g(r)\, \mathrm dr} + \frac{1}{C^2 + 1}.
	\end{equation*}
	Choosing $C = \alpha^{1/(2+d)}$ yields the claim.
	\end{proof}
	
%	For $\xi \in \mathbf N_f(\triangle)$ we define $\mathcal T(\xi)$ to be the maximal length on the real axis that can be covered by a subconfiguration of $\xi$ consisting of pairwise disjoint intervals, i.e.
%	\begin{align*}
%	\mathcal T(\xi) \coloneqq \sup \Big\{ \sum_{i=1}^{n}t_i - s_i :\,&n\in \N,\, (s_1, t_1), \hdots, (s_n, t_n) \in \operatorname{supp}(\xi) \\
%	&\text{ such that }[s_i, t_i] \cap [s_j, t_j] = \emptyset \text{ for }i\neq j \Big\} \vee 0.
%	\end{align*}
%	
	\section{Appendix: facts about Gaussian measures}
	In the following, we collect a few facts about Gaussian measures (which were in a similar form already applied in \cite{MV19}) in order to be self contained.
	\begin{lemma}
		\label{Lemma: phi as a determiant}
		Let $X = (X_1,\hdots, X_n)$ be a centered Gaussian vector with covariance matrix $C$. Then
		\begin{equation*}
			\mathbb E\Big[\prod_{i=1}^{n} e^{-X_i^2/2}\Big] = \frac{1}{\sqrt{\det(I_{n}+ C)}}.
		\end{equation*}
	\end{lemma}
	\begin{proof}
		Write $C = O \Lambda O^T$ with $O \in \R^{n\times n}$ orthogonal and $\Lambda \in \R^{n\times n}$ diagonal with non-negative entries. Then $C = A A^T$ with $A \coloneqq O \Lambda^{1/2}$.
		Let $Z = (Z_1, \hdots, Z_n) \sim \mathcal N(0, 1)^{\otimes n}$. Then $AZ \overset{d}{=} X$ and thus
		\begin{align*}
			\mathbb E\Big[\prod_{i=1}^{n} e^{-X_i^2/2}\Big] = \mathbb E\Big[e^{-|X|^2/2}\Big] &= \frac{1}{(2 \pi)^{n/2}} \int_{\R^n} e^{-|Az|^2/2} e^{-|z|^2/2} \mathrm dz \\
			&= \frac{1}{(2 \pi)^{n/2}} \int_{\R^n} e^{-\sum_{i=1}^n (1 + \lambda_i )z_i^2/2} \mathrm dz \\
			&= \prod_{i=1}^n (1+ \lambda_i)^{-1/2}.
		\end{align*}
		Since $\prod_{i=1}^n (1+ \lambda_i) = \det(I + \Lambda)= \det(I + C)$ the statement follows.	
	\end{proof}

	\begin{lemma}
		\label{Lemma: Determiant in terms of L2 distance}
		Let $X = (X_1,\hdots, X_n)$ be a centered Gaussian vector with covariance matrix $C$. Then
		\begin{equation*}
			\det(C) = \prod_{k=1}^{n} \operatorname{dist}_{L^2}(X_k, \operatorname{span}\{X_1, \hdots, X_{k-1}\})^2.
		\end{equation*}
	\end{lemma}
	
	\begin{proof}
		Write $X_n = Y_n + Z_n$ with $Y_n \in \operatorname{span}\{X_1, \hdots, X_{n-1}\}$ and $Z_n \perp \operatorname{span}(X_1, \hdots, X_{n-1})$. Define $Y_i \coloneqq X_i$ for $1\leq i \leq n-1$. Let $C'$ and $C''$ be the covariance matrices of $(Y_1, \hdots, Y_{n-1})$ and $(Y_1, \hdots, Y_{n})$ respectively. Then, by the Leibniz formula for the determinant
		\begin{align*}
			\det(C) &= \sum_{\sigma \in S_n}  \operatorname{sgn}(\sigma) \prod_{i=1}^{n} \mathbb E[X_{i} X_{\sigma(i)}] \\
			&= \mathbb E[X_n^2]\sum_{\sigma: \sigma(n) = n} \operatorname{sgn}(\sigma)\prod_{i=1}^{n-1} \mathbb E[X_{i} X_{\sigma(i)}] + \sum_{\sigma: \sigma(n)\neq n }  \operatorname{sgn}(\sigma) \prod_{i=1}^{n} \mathbb E[X_{i} X_{\sigma(i)}] \\
			&= \big(\mathbb E[Y_n^2] + \mathbb E[Z_n^2]\big)\sum_{\sigma:\sigma(n) = n}  \operatorname{sgn}(\sigma)\prod_{i=1}^{n-1} \mathbb E[Y_{i} Y_{\sigma(i)}] + \sum_{\sigma: \sigma(n)\neq n }  \operatorname{sgn}(\sigma) \prod_{i=1}^{n} \mathbb E[Y_{i} Y_{\sigma(i)}] \\
			&= \mathbb E[Z_n^2] \det(C') + \det(C'').
		\end{align*}
		Since $Y_n \in \operatorname{span}\{Y_1, \hdots, Y_{n-1}\}$ we have $\det(C'') = 0$ and the statement follows inductively.
	\end{proof}
	\begin{cor}
		\label{Corollary: Factor out of phi}
		For any centered Gaussian vector $X= (X_1, \hdots, X_n)$ we have
	\begin{equation*}
		\mathbb E\Big[\prod_{i=1}^{n} e^{-X_i^2/2}\Big] \geq \frac{1}{\sqrt{1 + \mathbb E[X_n^2]}}  \mathbb E\Big[\prod_{i=1}^{n-1} e^{-X_i^2/2}\Big].
	\end{equation*}
	\end{cor}
	\begin{proof}
		Let $Z = (Z_1, \hdots, Z_n) \sim \mathcal N(0, 1)^{\otimes n}$ be independent of $X$. Then the covariance matrix of $Z+X$ is $I+C$, where $C$ is the covariance matrix of $X$.
		By the previous two lemmas, 
		\begin{equation*}
			\mathbb E\Big[\prod_{i=1}^{n} e^{-X_i^2/2}\Big] = \frac{\mathbb E\Big[\prod_{i=1}^{n-1} e^{-X_i^2/2}\Big]}{\operatorname{dist}_{L^2}\big(X_n + Z_n, \operatorname{span}\{X_1+ Z_1, \hdots, X_{n-1}+Z_{n-1}\}\big)}.
		\end{equation*}
		The statement follows from
		\begin{align*}
			\operatorname{dist}_{L^2}(X_n + Z_n, \operatorname{span}\{X_1+ Z_1, \hdots, X_{n-1}+Z_{n-1}\} &\leq \mathbb E[(X_n + Z_n)^2]^{1/2} \\
			&= \sqrt{1+ \mathbb E[X_n^2]}. \qedhere
		\end{align*}
	\end{proof}


\begin{thebibliography}{BHL+02}
	
	\bibitem[BP21]{BP21}
	Volker Betz and Steffen Polzer.
	\newblock A functional central limit theorem for Polaron path measures.
	\newblock {\em arXiv: 2106.06447}, 2021	
	
	\bibitem[BMPV22]{BMPV22}
	Volker Betz, Chiranjib Mukherjee, Steffen Polzer and S.~R.~S. Varadhan.
	\newblock Polaron path measures with attractive interaction.
	\newblock {\em In preparation}
	
	\bibitem[DS20]{DySp20}
	Wojciech Dybalski and Herbert Spohn.
	\newblock Effective mass of the polaron{\textemdash}revisited.
	\newblock {\em Annales Henri Poincar{\'{e}}}, 21(5):1573--1594, 2020.
	
	\bibitem[DV83]{DoVa83}
	M.~D. Donsker and S.~R.~S. Varadhan.
	\newblock Asymptotics for the polaron.
	\newblock {\em Communications on Pure and Applied Mathematics}, 36(4):505--528,
	July 1983.
	
	\bibitem[FV17]{FV17}
	Sacha Friedli and Yvan Velenik.
	\newblock Statistical Mechanics of Lattice Systems: A Concrete Mathematical Introduction.
	\newblock {\em Cambridge University Press}, 2017
	
	\bibitem[LP48]{LP48}
	L.~D. Landau and S.~I. Pekar
	\newblock Effective mass of a polaron
	\newblock {\em Effective Mass of a Polaron, Zh. Eksp. Teor. Fiz. 18, 419–423, 1948}
	
	\bibitem[LT97]{LiTh97}
	Elliott~H. Lieb and Lawrence~E. Thomas.
	\newblock {Exact ground state energy of the strong-coupling polaron}.
	\newblock {\em Communications in Mathematical Physics}, 183(3):511 -- 519,
	1997.
	
	\bibitem[LS20]{LiSe20}
	Elliott~H. Lieb and Robert Seiringer.
	\newblock {Divergence of the Effective Mass of a Polaron in the Strong Coupling Limit}.
	\newblock {\em J Stat Phys}, 180, 23–33, 2020
	
	\bibitem[Moe06]{Moe06}
	Jacob Schach Møller.
	\newblock {The Polaron revisited}.
	\newblock {\em Reviews in Mathematical Physics}, 18, Vol. 18, No. 5 (2006) 485–517.
		
	\bibitem[MV19]{MV19}
	Chiranjib Mukherjee and S.~R.~S. Varadhan.
	\newblock Identification of the polaron measure i: Fixed coupling regime and
	the central limit theorem for large times.
	\newblock {\em Communications on Pure and Applied Mathematics}, 73(2):350--383,
	August 2019.
	
	\bibitem[MV21]{MV21}
	Chiranjib Mukherjee and S.~R.~S. Varadhan.
	\newblock Identification of the polaron measure i: Fixed coupling regime and
	the central limit theorem for large times.
	\newblock {ArXiv: 1802.05696},
	September 2021.
	
	\bibitem[MS21]{MS21}
	Krzysztof Myśliwy and Robert Seiringer.
	\newblock {Polaron Models with Regular Interactions at Strong Coupling}.
	\newblock {Journal of Statistical Physics}, 186, 2021.
	
	\bibitem[Pek49]{Pe49}
	S.I. Pekar.
	\newblock {Theory of polarons}.
	\newblock {\em Zh. Eksperimen. i Teor. Fiz.}, 19, 1949.
	
	
	\bibitem[Spo87]{Sp87}
	Herbert Spohn.
	\newblock Effective mass of the polaron: A functional integral approach.
	\newblock {\em Annals of Physics}, 175(2):278--318, May 1987.
	
\end{thebibliography}
\end{document}